\documentclass[12pt,wide,reqno]{amsart}

\usepackage{graphicx}
\usepackage{amsmath}
\usepackage{amssymb}
\usepackage{amscd}
 \usepackage{tikz}
\usepackage{pdfpages}
\usepackage{bbm}
\usepackage{amsthm}
\usepackage{amsfonts}
\usepackage{cite}

\newtheorem{theorem}{Theorem}

\newtheorem{theoremc}{Theorem}

\newtheorem{prop}[theoremc]{Proposition}

\newcommand\bib[1]{\bibitem[#1]{#1}}
\newcommand{\comm}[1]{}

\newcommand\C{{\mathbb C}}

\newcommand\op[1]{\mathop{\rm #1}\nolimits}
\newcommand\ot{\otimes}
\newcommand\p{\partial}

\newcommand\R{{\mathbb R}}

\newcommand\vp{\varphi}

\newcommand\z{\sigma}
\newcommand\Z{{\mathbb Z}}

\begin{document}

 \title[Submaximal rank 2 distributions in 5D]{On the models of submaximal\\
symmetric rank 2 distributions in 5D}
 \author{Boris Doubrov, Boris Kruglikov}
 \address{BD: Belarusian State University, Nezavisimosti ave. 4, 220030, Minsk, Belarus.
 \quad E-mail: {\tt doubrov@bsu.by}.
 \newline
 \hphantom{W} BK: \ Institute of Mathematics and Statistics, University of Troms\o, Troms\o\ 90-37, Norway.
 \quad E-mail: {\tt boris.kruglikov@uit.no}. }

 \begin{abstract}
There are two different approaches to exhibit submaximal symmetric rank 2 distributions in 5D
via Monge equations. In this note we establish precise relations between these models,
find auto-equivalences of one family, and treat two special equations.
 \end{abstract}

 \maketitle

\section{Introduction}

In the seminal paper \cite{C} E.\,Cartan proved that the maximal dimension of the symmetry
group of a rank 2 distribution in 5D is 14, whenever the distribution is completely non-holonomic
(bracket-generating) and is not the second prolongation
of the contact distribution in 3D (or not the first prolongation of the Engel distribution in 4D).
These requirements are equivalent to the growth vector being $(2,3,5)$ - we assume this from now on;
we also assume throughout that the manifold under consideration is connected.
The maximal symmetric model is (locally) unique.

Cartan also proved that the next (submaximal) dimension of the symmetry group is 7
(this is done under implicit assumption of the constancy of root type for the fundamental quartic
invariant, but this requirement can be removed \cite{KT}). Moreover he found that
these possess exactly 1 invariant $I$ (in fact, $I$ is not an invariant, but $I^2$ is; see \cite{K$_2$})
and classified all possible submaximal symmetric models.
He realized them as Monge equations \cite{G},
i.e. underdetermined ODEs consisting of 1 equation on 2 functions
(in fact, in \cite{C} only scalar second order PDEs with respect to contact transformations
are considered; Cartan proved equivalence of their geometry to that of rank 2 distributions, but he did not
consider Monge equations explicitly in this paper; 
the relation is however not difficult to uncover \cite{AK,K$_1$}).

One family of these Monge equations is especially simple \cite[p.113]{C}
 \begin{equation}\label{Pm}
y'=(z'')^m.
 \end{equation}
Here $m=0,1$ correspond to linear equations, which have infinite-dimensional symmetry algebra;
$m=-1,\frac13,\frac23,2$ (it is better to express this in terms of the parameter $k=2m-1$,
yielding $k=\pm3^{\pm1}$) correspond to $G_2$-symmetry;
all other parameters give 7-dimensional solvable symmetry algebra.

This family however misses 1 submaximal symmetric model 
(beside this it is complete over $\C$ but not over $\R$), while
the following family of models is complete (over both fields $\C$ and $\R$,
and we focuss on the latter):
 \begin{equation}\label{Qm}
y'=(z'')^2+r_1(z')^2+r_2z^2.
 \end{equation}
Here $r_1,r_2$ are arbitrary real constants, defined up to
transformation $r_1\mapsto cr_1$, $r_2\mapsto c^2r_2$. The value $r_2=\frac9{100}r_1^2$ corresponds
to $G_2$-symmetry, otherwise the symmetry group is 7-dimensional solvable.
This family arises in general on p.113 of \cite{C} and also on p.171 loc.cit. with the special values $r_1=\frac{10}3I$, $r_2=1+I^2$ ($I$ is constant for a 7D symmetry; also Cartan's
erroneous constant $\frac56$ is replaced
here to the correct value $\frac{10}3$ according with the observation of \cite{S}).

Cartan's invariant equals $I^2=\frac{(k^2+1)^2}{(k^2-9)(1/9-k^2)}$ via $k=2m-1$ for the family (\ref{Pm})
and $I^2=\bigl(\frac{100r_2}{9r_1^2}-1\bigr)^{-1}$ for the family (\ref{Qm}).
Note also that the semi-invariant $I$ can be real or imaginary, i.e. the invariant $I^2$ is
allowed to be negative (and also $\infty$), so following \cite{K$_2$} one should rather use
the invariant $25J=9(1+I^{-2})$ taking values in $\R_+$.

These two families were later studied in respectively \cite{K$_2$} and \cite{DZ$_2$}
(also in \cite{N}), but an exact relation between them was not written down in the literature.
In fact, in the first cited paper, it was even observed that
the parametrization of the bulk of the submaximal structures via (\ref{Pm}) is
4-fold, with the equivalent cases corresponding to $k\mapsto \pm k^{\pm1}$. However this
correspondence was not made explicit. In this paper we write the equivalence between all the models
explicitly and explain how we obtained the transformations.

\section{Submaximal (2,3,5) distributions: Monge equations I}\label{S2}

Let us start by exploring the submaximal symmetric family $P_m$ given by
(\ref{Pm}). We write the symmetries for the non-exceptional parameters of the above
Monge equation explicitly \cite{K$_2$} (we are using the jet coordinates $z_1$ instead of $z'$
and $z_2$ instead of $z''$):
 \begin{multline*}
\hspace{-12pt}\op{sym}(P_m)=\langle
W_1=\p_x, \ W_2=\p_y, \ W_3=\p_z, \ W_4=x\p_x+y\p_y+2z\p_z+z_1\p_{z_1},\\
W_5=x\p_z+\p_{z_1}, \ W_6=my\p_y+z\p_z+z_1\p_{z_1}+z_2\p_{z_2}, \ W_7=\\
\hspace{10pt}=
z_2^{m-1}\p_x+(m-1)\!\int\!z_2^{2m-2}dz_2\cdot\p_y+(z_1z_2^{m-1}-\tfrac1my)\p_z+(1-\tfrac1m)z_2^m\p_{z_1}
\rangle.
 \end{multline*}
Below and in what follows we use the notation $W_6'=W_6-\frac12W_4$.
The symmetry algebra is solvable and has the structure
$\mathfrak{m}^7(m)=\mathfrak{n}^5\rtimes\R^2$, where
$\mathfrak{n}^5=\mathfrak{h}_{-1}\oplus\mathfrak{h}_{-2}$ is the Heisenberg algebra
(indices denote the grading) given by the symplectic form on
$\mathfrak{h}_{-1}=\langle W_1,W_2,W_5,W_7\rangle$ with values in $\mathfrak{h}_{-2}=\langle W_3\rangle$.
The only non-trivial brackets are
 $$
[W_1,W_5]=W_3, \ [W_2,W_7]=\tfrac{-1}mW_3
 $$
(we keep the normalizing factor). This is right-extended (extension via derivations)
by $\R^2=\langle W_4,W_6'\rangle$ via the grading element $W_4$, $\op{ad}(W_4)|_{\mathfrak{h}_k}=k\cdot\op{id}$,
and $W_6'$ given by the action ($\theta_i$ is the coframe dual to $W_i$)
 $$
\op{ad}(W_6')=\tfrac12W_1\ot\theta_1+(\tfrac12-m)W_2\ot\theta_2
-\tfrac12W_5\ot\theta_5+(m-\tfrac12)W_7\ot\theta_7;
 $$
for $m=1/2$, when the spectrum of this operator on $\mathfrak{n}^5$ is multiple,
we add $+\frac12W_2\ot\theta_7$ to the above expression, introducing the Jordan block
(in the normal form of $\op{ad}(W_6')|_{\mathfrak{h}_{-1}}$ the sizes of blocks are $(1,1,2)$).

It is easy to see that, in the case the symmetry algebra is transitive, the
symmetry algebra together with the filtration induced by the distribution and isotropy
subalgebra (see \cite{K$_2$}) is the complete invariant of the geometric structure
(= 2-distribution encoding the equation). This implies that the conformal spectrum of
$\op{ad}(W_6')$ is the complete invariant, leading to the numerical invariant \cite{K$_2$}
$J(m)=\frac{(1-2m)^2}{(1-2m+2m^2)^2}$.

Consequently the number of elements in the equivalence class of a generic element of the family $P_m$ 
is 4 (we encode $P_m$ as a the rank 3 Pfaffian system on $M^5=\R^5(x,y,z,z_1,z_2)$;
its transformations correspond to internal equivalences of the Monge equations).

The group $G=\Z_2\times\Z_2$ with generators $a$, $b$ (in the 1st and 2nd copy
of $\Z_2$ respectively; $a^2=b^2=1$, $ab=ba$) acts on the open subset 
 $$\R^\times=\R\setminus\{\tfrac12\}\subset\R$$ 
in the parameter space $\R=\R(m)$, 
and also on its compactification $\R P^1=\R\cup\infty$, as follows:
 $$
a\cdot m=\bar{m}=1-m,\ b\cdot m=\mu=\frac{m}{2m-1},\ (ab)\cdot m=\bar{\mu}=\frac{m-1}{2m-1},
 $$
and the orbit consists of equivalent parameters: $P_m\simeq P_{g\cdot m}$, $g\in G$
(this equivalence is indeed not unique, as there is a 7D group of
symmetries of $P_m$, which is bigger than the Lie group generated by $\exp(\mathfrak{m}^7)$).

More precisely, we claim

 \begin{prop}
Systems $P_m$ and $P_n$ are equivalent if and only if $n=g\cdot m$, where $n,m\in\R^\times$
(neither of them is equivalent to $P_{\frac12}$).
An equivalence $T_a$ mapping $P_m$ to $P_{\bar{m}}$ is given by the Legendre transform:
 $$
T_a(x,y,z,z_1,z_2)=(z_1,y,xz_1-z,x,1/z_2).
 $$
An equivalence $T_b$ mapping $P_m$ to $P_\mu$ is given by the following formula
 \begin{multline*}
T_b(x,y,z,z_1,z_2)=\Bigl(\tfrac{1}{\sqrt{2m-1}}xz_2^{1-m}, \
\tfrac{m}{\sqrt{2m-1}}z_1-\tfrac{m-1}{\sqrt{2m-1}}xz_2, \\ \qquad z-xz_1+\tfrac{1}{m}xyz_2^{1-m}+\tfrac{(m-1)^2}{m(2m-1)}x^2z_2, \
\tfrac{\sqrt{2m-1}}{m}y-\tfrac{m-1}{m\sqrt{2m-1}}xz_2^m, \ z_2^{2m-1}\Bigr).
 \end{multline*}
An equivalence mapping $P_m$ to $P_{\bar{\mu}}$ can be taken as the composition $T_aT_b$.
 \end{prop}

The proof is a straightforward computation. Notice that the action of $G$ on $\R$ in the
parameter $k=2m-1$ is simpler, $a(k)=-k$, $b(k)=k^{-1}$ (the latter is not defined at 0,
which corresponds to $m=1/2$). 

Consider the symmetry group of the square $\tilde G=\op{Dih}_4$. It has 8 elements,
can be included into exact sequence $1\to\Z_4\to\tilde G\to\Z_2\to1$,
and admits a surjective homomorphism $p:\tilde{G}\to G$ with $\op{Ker}(p)=\Z_2$.

 \begin{theorem}
Equip the bundle ${\R^\times}\times M$ over $\R^\times$ with the distribution
$P_m$ in the fiber $M_m$ over $m\in\R^\times$. There is a local action $T$ of the group
$\tilde G$ on ${\R^\times}\times M$ covering the action of $G$ on $\R^\times$ such that $T$
preserves the distribution in the fibers, i.e. for every $g\in\tilde G$, 
$T_g:M_m\to M_{p(g)\cdot m}$ 
maps the
Pfaffian system corresponding to $P_m$ to that of $P_{p(g)\cdot m}$.
 \end{theorem}

 \begin{proof}
The Legendre transform is indeed an involution $T_a^2=\op{Id}$.
One can also check that $T_b^2=\op{Id}$. However $T_aT_b\neq T_bT_a$.
Denote $\z=(T_aT_b)^2$, then we have $T_aT_b=T_bT_a\z$ and $\z^2=\op{Id}$.

The group generated by $a,b$ is $\op{Dih}_4$: $a$ corresponds to reflection with respect to
the axis passing through the center of the square and the mid-point of one side,
$b$ corresponds to reflection with respect to a diagonal. Then $\z$ corresponds to reflection
with respect to the center of the square ($\z a=a\z$ is the reflection with respect to the perpendicular
axis, and $\z b=b\z$ is the reflection with respect to the other diagonal).

The subgroup $\Z_4\subset\tilde{G}$ generated by $ab$ is normal, the quotient
$\tilde{G}/\Z_4\simeq\Z_2$ is generated by $\z$.
The projection $p:\tilde{G}\to G$ is given by $p(a)=a$, $p(b)=b$.
Thus we get the required action of $\tilde{G}$ on $M^5$ that is intertwined via $p$
with the action of $G$ on $\R^1$.
 \end{proof}

{\bf Remark.} There are 2 (conjugated) subgroups $\Z_2\times\Z_2\simeq G_{(i)}\subset\tilde{G}$,
namely $G_{(1)}=\langle a,\z\rangle$ and $G_{(2)}=\langle b,\z\rangle$. Both project to $\Z_2\subset G$
via $p$, so they cannot be used to reduce $\tilde{G}$ to $G$.
It is not clear if there is a local action of $G$ on $\R\times M^5$,
preserving the distributions in fibers and covering the action of $G$ on $\R^\times$.

\smallskip

The anomaly measuring non-commutation of $T_a$ and $T_b$ is given by the formula
 \begin{multline*}
T_\z(x,y,z,z_1,z_2)=
\Bigl(\tfrac{2m-1}{m-1}yz_2^{1-m}-\tfrac{m}{m-1}z_1, \ yz_2^{1-2m}, \ \\
\tfrac1{\mu}(y^2z_2^{1-2m}-xyz_2^{1-m})+xz_1-z, \ \tfrac1{\mu}yz_2^{-m}-\tfrac{m-1}mx, \ z_2^{-1}\Bigr).
 \end{multline*}

\medskip

Let us explain how the equivalence of the models was obtained.
The idea is that since the symmetry algebra is transitive, the filtration on it, induced
by the stabilizer subalgebra and the distribution, is the complete invariant of the
geometric structure (distribution). Choosing a basis (generators of the transitive part)
of the vector fields on $M$ among $\op{sym}(P_m)$ and decomposing the other fields
(elements of the stabilizer) from $\op{sym}(P_m)$ via them, we obtain
the functions-coefficients, which serve as the invariants of the problem.
For an equivalent distribution we have to match the bases, and then equality of the
invariants yields the equivalence (up to symmetry).
In \cite{DK} a criterion for sufficiency of these invariants is given.

In the above model, we choose $W_1,W_2,W_3,W_5,W_6'$ as a basis of vector fields on $M^5$
and decompose $W_4,W_7$ via them. In fact, the decomposition
 $$
W_4=xW_1+yW_2+(2z-xz_1)W_3+z_1W_5
 $$
suffices for our purposes, since knowing the $(x,y,z)$-components of the equivalence
$w\mapsto \bar{w}=T(w)$, $w=(x,y,z,z_1,z_2)$, namely
$\bar{x}=\psi(w)$, $\bar{y}=\vp(w)$, $\bar{z}=\phi(w)$,
we can restore the rest by the formulae $\bar{z}_1=\frac{D\phi}{D\psi}=\phi_1$,
$\bar{z}_2=\frac{D\phi_1}{D\psi}$, $\bar{y}_1=\frac{D\vp}{D\psi}$,
where $D=\p_x+f\p_y+z_1\p_z+z_2\p_{z_1}+z_3\p_{z_2}$,
$y_1=f(w)$ is the original Monge equation, and $\bar{y}_1=\bar{f}(\bar{w})$
is the transformed Monge equation ($z_3$ has to cancel in the computation).

For instance, passing from $m$ to $1-m=a\cdot m$ we have to adjust the bases to
match the structure equations. This is achieved by the choice $\bar{W}_1=W_5$, $\bar{W}_2=W_2$,
$\bar{W}_3=-W_3$, $\bar{W}_4=W_4$, $\bar{W}_5=W_1$, $\bar{W}_6'=-W_6'$, $\bar{W}_7=\frac{m-1}mW_7$,
and this yields the Legendre transform $T_a$. Similarly, we obtained the formula for $T_b$ etc.

\section{Submaximal (2,3,5) distributions: Monge equations II}\label{S3}

Now consider the family (\ref{Qm}). Following \cite{DZ$_2$} it is more convenient to write
it in the following form:
 \begin{equation}\label{Qq}
y'=(z'')^2+(a^2+b^2)\,(z')^2+a^2b^2\,z^2
 \end{equation}
with the pair $[a:b]\in\C P^1$ (or $[a:b]\in\R P^1$ in the real case, but then this parametrization
does not cover all pairs $(r_1,r_2)$; however the case $[a:b]=\pm3^{\pm1}$ corresponding to $G_2$ symmetry
is real).

\smallskip

{\bf Remark.} The family (\ref{Pm}), united with the special Monge equation (\ref{ln}) below, 
gives a complete list of submaximal symmetric equations over $\C$, but not over $\R$. For instance, 
the Monge equation $(y')^2=1+\epsilon(z'')^2$ is not real equivalent to any of them for $\epsilon=+1$
(but for $\epsilon=-1$ it is equivalent to the Monge equation $P_{\frac12}$).
The family (\ref{Qm}) however yields the complete list of submaximal symmetric equations over both $\C$
and $\R$. It is clear that re-parametrization (\ref{Qq}) does not change generality over $\C$,
but over $\R$ it gives only those of equations from (\ref{Qm}) for which $r_2\geq0$, $r_1\geq2\sqrt{r_2}$.
Henceforth we restrict to either complex or real parameters for equation (\ref{Qq}).

\smallskip

The parameters $(a,b)$ can be considered up the action of $\op{Dih}_4$ generated by $(a,b)\mapsto (b,a), (-a,b), (a,-b)$ and up to a non-zero scale $(a,b)\mapsto (\lambda a, \lambda b)$, $\lambda\ne 0$. The group of scalings intersects with $\op{Dih}_4$ by its center $(a,b)\mapsto (\pm a,\pm b)$. So, we get the action of $G=\op{Dih}_4/\Z_2 = \Z_2 \times \Z_2$ when passing to the projective class $[a:b]\in P^1$.  In the affine chart $[1:\kappa]$ this action  reduces to  $\kappa\mapsto \pm \kappa$, $\kappa\mapsto \pm\kappa^{-1}$.

The quotient space $\C P^1/G$ (orbifold) is isomorphic to the half-disk: $\bar{D}_+=\{\kappa\in\C:|\kappa|\le1,\op{Re}(\kappa)\ge0\}$ and the two singular points (with orbits of cardinality 2) are $0,1\in\bar{D}$. They correspond to the points $[a:b]=[1:0],[1:1]\in\C P^1$, that will play a role in what follows.

Similarly, $\R P^1/G$ is an interval with end-points corresponding to singular
orbits and a marked point inside, corresponding to $G_2$-symmetry.

Denote Monge equation (\ref{Qq}) by $Q_{ab}$;
it is encoded as a rank 3 Pfaffian system on $\R^5(x,y,z,z_1,z_2)=N^5$.
Provided $a\pm b\neq0$, $a\cdot b\neq0$, $a:b\neq\pm3^{\pm1}$, its symmetries are
  \begin{multline*}
\hspace{-10pt}\op{sym}(Q_{ab})=\langle
U_6=\p_x, \ U_3=\p_y, \ U_4=2y\p_y+z\p_z+z_1\p_{z_1}+z_2\p_{z_2},\\
U_1=\xi(-a,-b), \ U_2=\xi(-b,-a), \ U_5=\xi(a,b), \ U_7=\xi(b,a)\rangle,
 \end{multline*}
where
 $$
\xi(a,b)=e^{bx}\bigl(\p_z+b\,\p_{z_1}+b^2\p_{z_2}+2b(a^2z+bz_1)\p_y\bigr).
 $$
The fields are numerated to match the structure equations of the algebra $\mathfrak{m}^7$
up to scaling, and these latter can be fixed provided we match the parameters $a=\pm\frac12$,
$b=\pm(m-\frac12)$ in (\ref{Qq}); since the signs play no role, we choose both "\!$+$" in what follows.

In other words, following the approach outlined at the end of the previous section,
we find that ODE (\ref{Pm}) is equivalent to the following Monge equation
 \begin{equation}\label{Qnm}
y'=(z'')^2+(m^2-m+\tfrac12)\,(z')^2+\tfrac14(m-\tfrac12)^2z^2.
 \end{equation}
Furthermore the suggested method implies the following result.
 \begin{prop}
An equivalence $\Psi:N^5\to M^5$ mapping (\ref{Qnm}) to (\ref{Pm}) is
given by the following formula (for $m\neq\tfrac12$)
 \begin{multline*}
\Psi(x,y,z,z_1,z_2)=
\Bigl(\bigl(\tfrac1mz_2+z_1+\tfrac{2m-1}{4m}z\bigr)e^{-\frac12x}, \
\tfrac1m(z_2-\tfrac14z)e^{(m-\frac12)x}, \\
\tfrac1{2m^2}\bigl(y-(m-\tfrac12)z{z_2}+2(m-1)z_1z_2+z_2^2+m(m-1)z_1^2-m(m-\tfrac12)z{z_1}\\
-\tfrac12(m^2-2m+\tfrac34)z^2\bigr), \
\tfrac1m(z_2+(m-1)z_1-\tfrac12(m-\tfrac12)z)e^{\frac12x}, \ e^x\Bigr).
 \end{multline*}
 \end{prop}

The choice $(a,b)=(1,0)$ in (\ref{Qq}), namely the Monge equation
 \begin{equation}\label{N12}
y'=(z'')^2+(z')^2
 \end{equation}
is equivalent to (\ref{Pm}) for $m=\frac12$
(notice that the above formulae for $\op{sym}(Q_{10})$ give only 6 of the symmetries since
$\xi(\pm1,0)=\p_z$; to these we shall add the 7th symmetry $\tilde{U}_7=\p_{z_1}+x\p_z+2z\p_y$).
The precise equivalence $\bar{\Psi}:N^5\to M^5$
between (\ref{N12}) and (\ref{Pm}) for $m=\frac12$ is given by
 $$
\bar\Psi(x,y,z,z_1,z_2)=
\Bigl((z_2-z_1)e^x,z_2-z,\tfrac{z_2^2-z_1^2}2+z_1z_2-y,(z_1+z_2)e^{-x},e^{-2x}\Bigr).
 $$

\medskip

Consider now a special Monge equation
 \begin{equation}\label{ln}
y'=\ln(z'').
 \end{equation}
This is the only underdetermined ODE with 7D symmetry algebra missing in the family
(\ref{Pm}), as was noticed in \cite{DG}. We shall show that it is equivalent to (\ref{Qq})
with the parameters $(a,b)=(1,1)$, namely to the Monge equation $Q_{11}$:
 \begin{equation}\label{NS}
y'=(z'')^2+2(z')^2+z^2;
 \end{equation}
notice that no parameter $m$ in (\ref{Qnm}) gives this ODE.

The symmetry algebra for the corresponding rank 2 distribution is
(notice that this is the 2nd singular case $a=\pm b$, so the formulae
do not follow from these of $\op{sym}(Q_{ab})$):
 \begin{multline*}
\hspace{-10pt}\op{sym}[(\ref{NS})]=\langle
U_2=e^x\bigl(2(z_1+z)\p_y+\p_z+\p_{z_1}+\p_{z_2}\bigr),
U_3=8\p_y, U_6=\p_x\\
U_1=e^x\bigl(2((x+1)z_1+(x-2)z)\p_y+(x-1)\p_z+x\p_{z_1}+(x+1)\p_{z_2}\bigr),\\
U_4=2y\p_y+z\p_z+z_1\p_{z_1}+z_2\p_{z_2},
U_5=e^{-x}\bigl(2(z_1-z)\p_y+\p_z-\p_{z_1}+\p_{z_2}\bigr),\\
U_7=e^{-x}\bigl(2((x-2)z_1-(x+1)z)\p_y+x\p_z-(x-1)\p_{z_1}+(x-2)\p_{z_2}\bigr)
\rangle.
 \end{multline*}
These are numerated to correspond to the structure equations in the numeration
of the symmetries of (\ref{ln}):
 \begin{multline*}
\hspace{-10pt}\op{sym}[(\ref{ln})]=\langle
V_1=-z_2^{-1}\bigl(\p_x+(\ln(z_2)+1)\p_y-(yz_2-z_1)\p_z\bigr)+(\ln(z_2)-1)\p_{z_1},\\
V_2=-2(x\p_z+\p_{z_1}),\ V_3=-2\p_z,\ V_4=x\p_x+y\p_y+2z\p_z+z_1\p_{z_1},\\
V_5=2\p_y,\ V_6=x\p_x+(y-2x)\p_y-z_1\p_{z_1}-2z_2\p_{z_2},\ V_7=\p_x\rangle
 \end{multline*}

The symmetry algebra of this exceptional case is similar to the general case,
the only difference is that in the normal form of $\op{ad}(W_6')|_{\mathfrak{h}_{-1}}$
(now we have $V_6$ instead of $W_6'$) the sizes of Jordan blocks are $(2,2)$. More precisely,
the above Lie algebra is solvable and has the structure
$\tilde{\mathfrak{m}}^7=\mathfrak{n}^5\rtimes\R^2$, where
$\mathfrak{n}^5=\mathfrak{h}_{-1}\oplus\mathfrak{h}_{-2}$ is the Heisenberg algebra
(indices denote the grading) given by the symplectic form on
$\mathfrak{h}_{-1}=\langle V_1,V_2,V_5,V_7\rangle$ with values in $\mathfrak{h}_{-2}=\langle V_3\rangle$.
The only non-trivial brackets are
 $$
[V_1,V_5]=V_3,\ [V_2,V_7]=-V_3.
 $$
This is right-extended (extension via derivations)
by $\R^2=\langle V_4,V_6\rangle$ via the grading element $V_4$, $\op{ad}(V_4)|_{\mathfrak{h}_k}=k\cdot\op{id}$,
and $V_6$ given by the action ($\theta_i$ is the coframe dual to $V_i$)
 $$
\op{ad}(V_6)=(V_1+V_2)\ot\theta_1+V_2\ot\theta_2-V_5\ot\theta_5+(V_5-V_7)\ot\theta_7.
 $$
Equalizing the coefficients of the decomposition of $U_4$ via $U_2,U_3,U_5,U_7$ with
the same for the fields $V_i$ we obtain an equivalence between the Monge equations
(\ref{ln}) and (\ref{NS}). Denote by $K^5=\R^5(x,y,z,z_1,z_2)$ the manifold-equation
corresponding to ODE (\ref{ln}).

 \begin{prop}
An equivalence $\Phi:N^5\to K^5$ mapping (\ref{NS}) to (\ref{ln}) is
given by the following formula
 \begin{multline*}
\Phi(x,y,z,z_1,z_2)=
\Bigl(
\tfrac12(z-z_2)e^x, \ \bigl(xz_2-z_1-(x-1)z\bigr)e^x, \\
\tfrac18\bigl(z_2(z_2+2z)-3z^2+4z_1z_2-2y\bigr), \ -\tfrac12(z_2+2z_1+z)e^{-x}, \ e^{-2x}
\Bigr).
 \end{multline*}
 \end{prop}

\medskip

Finally, we can also consider a similar to (\ref{ln}) Monge equation
 \begin{equation}\label{exp}
y'=\exp(z'').
 \end{equation}
Direct computations show that the symmetry algebra of this underdetermined ODE is
isomorphic to that of (\ref{Pm}) for $m=\frac12$. Indeed, the symmetries are
 \begin{multline*}
\hspace{-10pt}\op{sym}[(\ref{exp})]=\langle
V_1=\tfrac12\p_y, V_2=-(x\p_z+\p_{z_1}), V_3=-\tfrac12\p_z, V_7=\p_x,\\
V_4=x\p_x+y\p_y+2z\p_z+z_1\p_{z_1},
V_6=-\tfrac12(y\p_y+\tfrac12x^2\p_z+x\p_{z_1}+\p_{z_2}),\\
V_5=e^{z_2}\p_x+\tfrac12e^{2z_2}\p_y-(y-z_1e^{z_2})\p_z+(z_2-1)e^{z_2}\p_{z_1}\rangle
 \end{multline*}
(the numeration of the basis is adjusted, so that the structure equations coincide with these for
the equation (\ref{Pm}) with $m=\tfrac12$).

Thus there must be an equivalence, and following the developed algorithm we easily find it.
Denote by $L^5=\R^5(x,y,z,z_1,z_2)$ the equation-manifold of (\ref{exp}).

 \begin{prop}
An equivalence $\Upsilon:L^5\to M^5$ mapping (\ref{exp}) to (\ref{Pm}) for $m=\frac12$ is
given by the formula
 \begin{multline*}
\Upsilon(x,y,z,z_1,z_2)=\\
=\Bigl(2y-x\,e^{z_2}, x+z_1-xz_2, 2(xz_1-z)-x^2(z_2-\tfrac12), x\,e^{-z_2}, e^{-2z_2}\Bigr).
 \end{multline*}
 \end{prop}

We can compute an equivalence between (\ref{N12}) and (\ref{exp}).
It is given by the composition $T=\Upsilon^{-1}\circ\bar{\Psi}:N^5\to L^5$,
 \begin{multline*}
\!\!\!\!\!T(x,y,z,z_1,z_2)=
\Bigl(\exp(x+e^{-2x})(z_1-z_2)+2(z_2-z), 2z_1e^{-x}+(z_2-z_1)e^x,\\
\,\,\,2(y-z_1^2)+\tfrac12e^{2x}(z_1-z_2)^2,-\exp(x-e^{-2x})(z_1-z_2),\exp(-2e^{-2x})\Bigr).
 \end{multline*}

\section{Rank 2 distributions in higher dimensions}\label{S4}

The most symmetric rank 2 distribution in 6D, not reducible to a 2-distribution
in lower dimensions (this is equivalent to the growth vector (2,3,5,6)),
has symmetry algebra of dimension 11, and is isomorphic to the distribution
of the Monge equation $y'=(z''')^2$ \cite{DZ$_1$,AK}.

An obvious candidate for the submaximal symmetry dimension is $y'=(z''')^m$, 
which generalizes the Monge equations of type I in 5D. However its symmetry algebra is of dimension $\le8$,
while that for generalization of Monge equations of type II is of dimension 9.
Dimension 10 is not realizable, and so the submaximal dimension is 9 \cite{K$_2$}. 

A similar story happens in any dimension $d=n+3>5$. The maximally symmetric
Monge equation is $y'=(z^{(n)})^2$ (with additional condition
of nondegeneracy \cite{AK} or maximal class \cite{DZ$_1$}).
In \cite{DZ$_2$} it was shown that all submaximal rank 2 distributions with at least one non-vanishing Wilczynski invariant are given (both over $\R$ and $\C$) by the Monge equation
 $$
y'=(z^{(n)})^2+r_1(z^{(n-1)})^2+\dots+r_nz^2.
 $$
The structure of the singular strata (and marked points) of the corresponding
orbifold is similar to the considered for $n=2$. For example, over $\C$ the coefficients
$r_1,\dots,r_n$ can be defined via parameters $(a_1,\dots,a_n)\in\C^n$ as:
 \[
t^{2n} + r_1 t^{2n-2} + \dots + r_{n-1} t^2 + r_n = (t^2+a_1^2)(t^2+a_2^2)\cdots(t^2+a_n^2),
 \]
where $a_i$ are defined modulo permutations and arbitrary changes of sign. Algebraically this
defines a natural action of the group $\tilde{G}=S_n\rightthreetimes(\Z_2\times\dots\Z_2)$,
which is isomorphic to the Weyl group of the root system of type $C_n$.

In addition, $(a_1,\dots,a_n)$ is defined up to the non-zero scale, and the group of all scalings
intersects with $\tilde{G}$ by its center given by the map $(a_1,\dots,a_n)\mapsto \pm(a_1,\dots,a_n)$.
So passing to the projective point $[a_1:{\dots}:a_n]\in\mathbb{C}P^{n-1}$ we get the action
of the group $\tilde{G}/\Z_2$ on $\mathbb{C}P^{n-1}$.

The singular strata is defined as a set of all $[a_1:{\dots}:a_n]$ such that the set of
$2n$ numbers $\{\pm a_1,\dots,\pm a_n\}$ can be arranged into an arithmetic sequence
(this is stable with the respect to the above action of the group $\tilde{G}$).
This agrees with what was discussed for $n=2$.

For $n=2$ the factor $\tilde{G}/\Z_2$, as already noticed in Section \ref{S3},
coincides with the group $G=\Z_2\times\Z_2$ from Section \ref{S2}, and the isomorphism $\tilde{G}\simeq\op{Dih}_4$ (the Weyl group of
$C_2\simeq B_2$ is the symmetry group of the square) explains the mysterious appearance of
the dihedral group of order 8 for the Monge submaximal family I in Section \ref{S2}.


\end{document}